\documentclass[a4paper,12pt]{amsart}
\usepackage[latin1]{inputenc}
\usepackage[english]{babel}
\usepackage{amssymb,amsmath}
\usepackage{hyperref}
\usepackage{graphicx}
\usepackage{pifont,bbding}
\usepackage{amsfonts}
\usepackage{amsthm}
\usepackage{latexsym}

\usepackage{comment}

\newcommand{\R}{{\mathbb R}}

\newcommand{\N}{{\mathbb N}}

\newcommand{\al}{\alpha}

\newcommand{\Prol}{\mathrm{Prol}}

\renewcommand{\rho}{\varrho}
\renewcommand{\epsilon}{\varepsilon}
\renewcommand{\theta}{\vartheta}

\newcommand{\la}{\lambda}

\renewcommand{\a}{\alpha}

\newcommand{\ga}{\gamma}

\newcommand{\Z}{{\mathbb{Z}}}
\newcommand{\Rn}{{\mathbb{R}}^{n}}

\newcommand{\I}{{\mathcal{I}}}
\newcommand{\A}{\mathcal{A}}

\newcommand{\D}{{\mathcal{D}}}

\newcommand{\ex}{{\mathrm{e}}}

\newcommand{\Hom}{{\mathrm{Hom}}}
\newcommand{\Endo}{{\mathrm{End}}}

\newcommand{\barint}
{\rule[.036in]{.12in}{.009in}\kern-.16in \displaystyle\int}

\theoremstyle{plain}
\newtheorem{teo}{Theorem}[section]
\newtheorem{prop}[teo]{Proposition}
\newtheorem{defi}[teo]{Definition}
\newtheorem{lemma}[teo]{Lemma}

\theoremstyle{remark}

\newtheorem{remark}[teo]{Remark}

\title[Extremal polynomials]{Extremal polynomials in stratified groups}

\date{\today}

\author[Le Donne]{Enrico Le Donne}

\address[Le Donne]{Department of Mathematics and Statistics, P.O. Box 35,
FIN-40014,
University of Jyv\"askyl\"a, Finland}%
\email{ledonne@msri.org}

\author[Leonardi]{Gian Paolo
Leonardi}
\address[Leonardi]{Universit\`a di Modena e Reggio Emilia,
  Dipartimento di Scienze
Fisiche, Informatiche e Matematiche, via Campi 213/b, 41100
  Modena, Italy}

\email{gianpaolo.leonardi@unimore.it}

\author[Monti]{Roberto Monti}
\address[Monti and Vittone]{Universit\`a di Padova, Dipartimento di Matematica,
via Trieste 63, 35121 Padova, Italy}
\email{monti@math.unipd.it}

\author[Vittone]{Davide Vittone}
\email{vittone@math.unipd.it}



\thanks{This work was partially supported by 
INDAM,  University of Padova research project ``Some analytic and
differential geometric aspects in Nonlinear Control Theory, with
applications to Mechanics'', and  Fondazione CaRiPaRo Project
``Nonlinear Partial Differential Equations: models, analysis, and
control-theoretic problems''.}

\begin{document}

\begin{abstract}


We introduce a family of  {extremal polynomials} associated with the
prolongation of a stratified nilpotent Lie algebra. 
These polynomials are related to a new algebraic characterization of abnormal
subriemannian geodesics in   stratified nilpotent Lie groups. 
They satisfy  a set of remarkable 
structure relations that are used to integrate the  {adjoint equations}. 

\end{abstract}

\maketitle
 \section{Introduction}

In any stratified nilpotent Lie group, we compute explicitly the solutions
to the adjoint equations
for extremal curves related to the natural left invariant horizontal
distribution. In the   Hamiltonian formalism,
a normal extremal $(\gamma,\lambda)$, where $\ga$ is a curve in the group and
$\lambda$ is the dual curve in the cotangent bundle, solves the system of
equations
\begin{equation}\label{SH}
 \dot \gamma = \frac{\partial H}{\partial \lambda}(\gamma,\lambda),\quad
\dot\lambda = - \frac{\partial H}{\partial \gamma}(\gamma,\lambda),
\end{equation}
where $H$ is the Hamiltonian function. 
After fixing 
a basis of the Lie algebra of the group inducing 
exponential
coordinates of the second type, we integrate  
the second equation,  namely the equation $\dot\lambda = - {\partial H}/
{\partial \gamma}$ (see the precise formula in Theorem \ref{teo:eqdualvar1}),
and we express  $\lambda$ as a function of $\gamma$. 
In other words, we compute $n$ prime integrals of the
Hamiltonian
system \eqref{SH}, where $n$ is the dimension of the group.
The solutions are expressed in terms of a family of polynomials, called
{\em extremal polynomials},
that satisfy
a set of remarkable 
structure formulas (see \eqref{TH1} below). These formulas  involve the
structure constants of the Lie algebra of the group, in fact, of its Tanaka
prolongation. 
Extremal polynomials can also be   used to give an algebraic characterization of
{\em abnormal} extremals. This is the main motivation of our research.

\medskip

Let $G$ be a stratified nilpotent Lie group of dimension $n$ and rank $r$. The
Lie
algebra $\mathfrak g=\mathrm{Lie}(G)$ has the stratification $\mathfrak g =
\mathfrak
g_1\oplus\cdots\oplus \mathfrak g_s$,  where $s$ is the step of the
algebra, $\mathfrak g_i = [\mathfrak g_{i-1},\mathfrak g_1]$ for $i=2,\ldots,s$,
and
$\mathfrak g_i = \{0\}$ for $i>s$. The rank of $\mathfrak g$ is
$r=\mathrm{dim}(\mathfrak g_1)$.
Let $X_1,\ldots ,X_n$ be a basis of $\mathfrak g$ adapted to the
stratification. We identify the group $G$ with $\Rn$ via exponential coordinates
of the second
type induced by the basis $X_1,\ldots ,X_n$, 
and we identify $\mathfrak g$ with the corresponding
 Lie algebra of left invariant vector fields in $\Rn$.

Let $\Prol(\mathfrak g)=\bigoplus_{k\leq s} \mathfrak g_k$ be the Tanaka
prolongation of $\mathfrak g$, see \cite{T}. Even though this is not 
strictly needed in our argument, an explicit construction  is briefly recalled
in
Section
\ref{section:polyformulas}.  
We extend $X_1,\ldots ,X_n$   to a basis
$\{X_j\}_{j\leq n}$ of the prolongation 
and to each $j\leq n$ we assign the degree $d(j)= k$ if and only if $X_j
\in\mathfrak g_k$. 
Then we assume that the basis  is adapted to the graduation:
this means
that $i<j$ implies    $d(i)\leq d(j)$.
When the prolongation is finite dimensional, the index
$j$ ranges in a finite set,  $m \leq j\leq n$  for some
$m\in\Z$. With abuse of notation, we   denote the basis
$\{X_j\}_{m\leq j\leq n}$ by $\{X_j\}_{j\leq n}$, as
in the infinite dimensional case.
Let $c_{ij}^k \in\R$ be the structure constants 
of $\Prol(\mathfrak g)$ associated with the basis $\{X_j\}_{j\leq n}$.
Namely, for all $i,j\in\Z$ with $i,j\leq n$ we have 
\begin{equation}\label{StrCon2}
 [X_i,X_j] = \sum_{k\leq n}  c_{ij}^k X_k.
\end{equation}
The sum is always   finite, because each stratum $\mathfrak g_{k}$ 
of $\Prol(\mathfrak g)$ is finite dimensional and $c_{ij}^k=0$
if $d(k) \neq d(i)+d(j)$.

In this paper, we introduce a family of  \emph{extremal polynomials}
$P_j^v(x)$, $j\leq n$ and $x\in\R^n$,   associated with the
basis $\{X_j\}_{j\leq n}$ of $\Prol(\mathfrak g)$. They depend linearly on a
parameter $v\in\Rn$, see Definition \ref{Pollo}. 
Extremal polynomials satisfy the following  structure formulas.

\begin{teo}\label{teo1}
For
any $v\in \Rn$, $i=1,\dots,n$, and $j\in\Z$ with $j\leq n$ there holds 
\begin{equation}\label{TH1} 
X_i P_j^v (x) = \sum_{k\leq n} c_{ij}^k P_k^v(x),\quad x\in\Rn.
\end{equation}
\end{teo}

Modulo the value at $x=0$, extremal polynomials are uniquely determined by the
family of identities
\eqref{TH1}.
These structure formulas
are the core of the paper and of our main technical result,   Theorem
\ref{propderivpolinomi}. They first appeared in 
\cite{LLMV}, but only in the case of \emph{free} groups. In \cite{LLMV},
the formulas were obtained a posteriori, as a 
consequence of certain algebraic identities, 
however only for $j=1,\ldots,n$ with no reference to the prolongation and only 
for the coordinates related to a Hall basis.

Our interest in extremal polynomials origins 
in the regularity problem of sub-Riemannian length minimizing curves,
one of the main open problems in the field (see \cite{montgomery},
\cite{AgrSomeProb}, \cite{monti2}).
Let $M$ be a differentiable manifold and $\mathcal D\subset TM$ a bracket
generating distribution. A Lipschitz curve $\gamma:[0,1]\to M$ is 
horizontal  if $\dot\gamma(t) \in \mathcal D(\gamma(t))$ for a.e.~$t\in[0,1]$.
Fixing a quadratic form on $\mathcal D$, one can define the length of horizontal
curves. Length 
minimizing curves may be either normal extremals or abnormal extremals:
while normal extremals are always smooth, abnormal ones are apriori
only Lipschitz continuous.  Abnormal extremals
depend only on the structure $(M,\mathcal D)$:
they are precisely the singular points of the end-point mapping.

In Section \ref{section:alg-abn-extr}, we use extremal polynomials and
Theorem \ref{teo1} to 
give an
algebraic characterization
of abnormal  extremals in stratified nilpotent Lie groups (Carnot groups).
This is of special interest because, by 
Mitchell's theorem, 
Carnot groups are the infinitesimal model of equiregular sub-Riemannian
structures.

Let $\vartheta_1,\ldots,\vartheta_n$ be a basis of $1$-forms of $\mathfrak
g^*$, the dual of $\mathfrak g$.
A curve $\lambda \in\mathrm{Lip}([0,1];\mathfrak g^*)$ is given by coordinates
$\lambda_1,\ldots,\lambda_n\in\mathrm{Lip}([0,1])$ such that
$\lambda=\lambda_1 \vartheta_1+\ldots +\lambda_n\vartheta_n$.
If $\lambda$ is the dual curve of a normal or abnormal   extremal
$\gamma:[0,1]\to G=\Rn$,
then 
\begin{equation}\label{8121}
\dot\lambda_i = - \sum_{k=1}^n \sum_{j=1}^r  c_{ij}^k  \dot \gamma_j 
\lambda_k\quad\text{a.e.~on }[0,1],\quad i=1,\ldots,n.
\end{equation}
See Section
\ref{section:alg-abn-extr}  and Theorem \ref{teo:eqdualvar1}
for more details. In Theorem
\ref{teointduale}, we use the structure formulas \eqref{TH1} to integrate the
system of adjoint equations \eqref{8121}. The solutions are 
\[
 \lambda_i(t) = P^v_{i}(\gamma(t)),\quad i=1,\ldots,n,
\]
where $v\in\Rn$ is such that $v_i = \lambda_i(0)$. Thus, we can  prove the
following theorem
(see   Theorem \ref{cor:polin} for the complete statement).

\begin{teo}\label{ab}
Let $G=\Rn$ be a stratified nilpotent Lie group and let $\ga:[0,1]\to G$ be a 
horizontal curve with $\ga(0)=0$. Then, the following statements are
equivalent:
\begin{itemize}
 \item[(A)] The curve $\ga$ is an abnormal extremal.
 \item[(B)] There exist  $v\in\Rn$, $v\neq0$, such that $P_i^v(\gamma(t))=0$
for all $t\in[0,1]$ and for all $i\leq r$.
\end{itemize} 
\end{teo}

Abnormal extremals are precisely the horizontal curves 
lying inside algebraic 
varieties defined via extremal polynomials. 
This result extends   \cite[Theorem 1.1]{LLMV}  because  it applies
to {\em nonfree} Carnot
groups. It also improves   that result because  abnormal curves are
shown to be in
algebraic varieties {\em smaller}
than those considered in \cite{LLMV}. Notice that in (B) also indexes 
$i\leq 0$ related to the Tanaka prolongation are involved.

The role of   Tanaka prolongation in our theory is rather subtle.
We   searched for an integration algorithm inverting  the
differentiation process that now is established by the structure formulas
\eqref{TH1}.
For $i=1,\ldots,r$ and  $j=1,\ldots,n$, consider the integrals
\[
 B_{ij}^v(t) = \int_0^tP_i^v(\gamma(s))\dot\gamma_j(s) ds,\quad t\in[0,1].
\]
We may then further integrate $B_{ij}^v$ against $\dot\gamma_k$, etc. The
structure of the Tanaka prolongation tells us, however only from a deductive
point of view,
when the functions $B_{ij}^v(t)$ and their iterative integrals are
\emph{polynomials of the coordinates} $\gamma_1(t),\ldots,\gamma_n(t)$. 
Part of statement (B) in Theorem \ref{ab}
is that an abnormal extremal $\gamma$ satisfies $P^v_i(\gamma)=0$ for some
$v\neq0$ and for all $i=1,\ldots,r$.  
When the above integration process succeds, we get new polynomials
vanishing along the curve $\gamma$.
More details on this point of
view can be found in the example studied in Section \ref{esempi}.


In the final part of the paper, we show two applications of
the theory. 

In Section \ref{esempi}, we develop a technique to construct
a nontrivial algebraic set containing \emph{all} abnormal
extremals passing through one point.
This is related to the problem of estimating
the size of the set of regular values of the
end-point mapping (see \cite{AgrSomeProb} and \cite[Section 10.2]{montgomery}).
In sub-Riemannian manifolds with a distribution of corank
$1$, the image of the set of \emph{length minimizing}
abnormal extremals starting from one point  has 
zero Lebesgue measure; on the other
hand, independently from   corank, the image of \emph{strictly} abnormal 
 {length minimizing}
  extremals has 
empty interior (see Corollary 3 in \cite{RT} and  
\cite{A}). 

Our technique seems to work when 
the prolongation is sufficiently large.
In this case, for each abnormal curve $\gamma$ there is at least one
parameter $v\in\Rn$ and \emph{many} indexes $j$ such that $P_j^v(\gamma)=0$.
It is then possible to find a polynomial $Q$  independent of $v$ 
such that $Q(\gamma)=0$ for any abnormal curve $\gamma$ passing through one
fixed point.
We describe  the technique in detail in the case of the
free nilpotent Lie group of rank $2$ and
step $4$. However, it can be implemented in many other examples
and it is likely to work in any nonrigid group.

Finally, in Section \ref{Goh}  we construct a $64$-dimensional Lie group
possessing a
spiral-like Goh extremal whose tangents at the singular point
are all lines. This example points out  a
limitation of the
shortening technique introduced in \cite{LM} and developed in 
\cite{monti1} and \cite{monti3}. It is also interesting in relation
to the examples of \emph{nonrectifiable} spiral-like rigid paths studied in  
\cite{ZZ} and \cite{Z}, whereas our spiral-like extremal  has finite length.

 \medskip

\emph{Acknowledgements.} It is a pleasure to thank  
Ben Warhurst and Alessandro Ottazzi for many illuminating discussions on Tanaka prolongation.
We also thank Igor Zelenko for some discussions on a preliminary version of the
paper.

\section{Structure formulas for extremal polynomials}

\label{section:polyformulas}

Let $X_1,\ldots ,X_n$ be a basis of $\mathfrak g = \mathrm{Lie}(G)$ adapted to the
stratification. We identify the group $G$ with $\Rn$ via exponential coordinates
of the second
type induced by the basis $X_1,\ldots ,X_n$. Namely, for any $  x =
(x_1,\ldots,x_n) \in\Rn$ we have
\begin{equation} \label{ExpCord}
 x = \ex^{x_1 X_1}\circ \ldots \circ \ex^{x_n X_n}(0) =\exp(x_n
X_n)\cdot\ldots \cdot \exp(x_1 X_1).
\end{equation}
Above, $\exp:\mathfrak g\to G$ is the exponential mapping, $\cdot$ is the
group law in $G = \Rn$, 
and $\gamma(t) = \ex^{t X}(x)$, $t\in\R$, is the solution of the Cauchy Problem
\[
\begin{cases}
\dot\gamma = X(\gamma)\\ 
\gamma(0)=x.
\end{cases}
\]
In simply connected nilpotent groups,
the exponential mapping is a global
diffeomorphism. After the identification \eqref{ExpCord}, 
the Lie algebra $\mathfrak g$
is isomorphic to a Lie algebra of  vector fields in $\Rn$  that are left
invariant with respect to the group law $\cdot$.

We recall the construction of the Tanaka prolongation. 
First, we define the vector space of all strata-preserving derivations of
$\mathfrak g$:
\[
 \mathfrak g_0 = \Big\{ \phi \in \bigoplus_{i=1}^s
\mathfrak g_{i}\otimes \mathfrak g_{i}^{*}:
\phi([X,Y])=[\phi(X),Y]+[X,\phi(Y)],\,\,X,Y\in\mathfrak g\Big\}.
\]
Recall that $\mathfrak g_{i}\otimes \mathfrak g_{i}^{*}$ is
canonically isomorphic to $\Endo(\mathfrak g_{i})$. The direct sum vector space
$\bigoplus_{i=0}^s \mathfrak
g_i$ is a graded Lie algebra with the bracket
\[
 [\phi,X] = -[X,\phi]= \phi(X) ,\quad \textrm{for all $\phi \in\mathfrak g_0$
and $X\in\mathfrak g$},
\]
and with the natural bracket on  $\mathfrak g$ and
$\mathfrak g_{0}$.

By induction, assume that we have a vector space
$\bigoplus_{i=1-k}^s \mathfrak
g_i$ for
some $k\geq 1$ and assume that  the bracket $[\phi, X]=\phi(X)$ is already
defined for all $\phi \in \bigoplus_{i=1-k}^s \mathfrak
g_i$ and $X\in \mathfrak g$. Then we define the vector
space of all derivations $\phi : \mathfrak g \to \mathfrak
g_{1-k}\oplus\ldots \oplus \mathfrak g_{s-k}$ such that $\phi( \mathfrak
g_i)\subset   \mathfrak g_{i-k} $:
\[
 \mathfrak g_{-k}  = \Big\{ \phi \in \bigoplus_{i=1}^s
\mathfrak g_{i-k}\otimes \mathfrak g_{i}^{*}:
\phi([X,Y])=[\phi(X),Y]+[X,\phi(Y)],\,\,X,Y\in\mathfrak g\Big\}.
\]
 Recall that $\mathfrak g_{i-k}\otimes \mathfrak
g_{i}^{*}$ is isomorphic to
$\Hom(\mathfrak g_{i} ; \mathfrak  g_{i-k})$. As above, 
the direct sum vector
space 
$\bigoplus_{i=-k}^s \mathfrak
g_i$
is a graded Lie algebra with the bracket
\[
 [\phi,X]  = -[X,\phi]= \phi(X),\quad \textrm{for all $\phi \in\mathfrak g_{-k}$
and $X\in\mathfrak g$},
\]
with the natural bracket on  $\mathfrak
g$, and with the bracket $[\phi,\psi]$ for $\phi,\psi\in
\mathfrak g_{-k}\oplus\ldots\oplus\mathfrak g_0$ defined by
\[
  [\phi,\psi] (X) = [[\phi,X],\psi] + [\phi,[\psi,X]],\quad  X\in\mathfrak g. 
\]

This inductive construction may or may not end after a finite number of steps,
i.e., we have either $\mathfrak g_{-k} = \{0\}$ for some $k\geq 1$ or $\mathfrak
g_{-k}
\neq \{0\}$ for all $k\geq 1$.
In both cases, we let 
\begin{equation}\label{PRO}
 \Prol(\mathfrak g) = \bigoplus_{k\leq s} \mathfrak g_k.
\end{equation}
$\Prol(\mathfrak g)$ is a \emph{graded}  Lie algebra, called Tanaka prolongation
of $\mathfrak g$. Namely, we have  
\begin{equation}
 \label{GRADO}
[\mathfrak g_i,\mathfrak g_j] \subset \mathfrak g_{i+j}, \quad \textrm{for all
$i,j\in\Z$ with $i,j,i+j\leq s$}.
\end{equation}
This is the unique property that we need in the proof of the structure theorem
of extremal polynomials, Theorem \ref{propderivpolinomi}.
In fact, we need \eqref{GRADO} only for $j=1,\ldots,s$.
We do not specifically need
the   Tanaka prolongation but only a graded Lie algebra extending
$\mathfrak
g$ and satisfying
\eqref{PRO} and \eqref{GRADO}. Among all such extensions 
$\Prol(\mathfrak g)$ is the largest one.

In general, the explicit computation
of $\Prol(\mathfrak g)$ is   difficult.
When $\mathfrak g$ is a \emph{free} nilpotent 
Lie algebra of step $s\geq 3$, then we have
$\Prol(\mathfrak g)=\mathfrak g_0\oplus\mathfrak g$,
with the exceptional case of the step $3$ and rank $2$ free Lie algebra (see
\cite{BEN}).

Let $\{X_j\}_{j\leq n}$ be a basis of $\Prol(\mathfrak g)$.
For integers ${j_0}, {j_1}, {j_2},   \dots,{j_k}\leq n$, we let
\begin{equation}\label{notazione 1}
 [X_{j_0},X_{j_1},X_{j_2}, X_{j_3}, \dots,X_{j_k}]
 = [\cdots[[[X_{j_0},X_{j_1}],X_{j_2}],X_{j_3} ],\dots,X_{j_k}].
\end{equation}
Then, for $\alpha = (\alpha_1, \ldots,
\alpha_n)\in \I=\N^n=\{0,1,2,\ldots\}^n$, we define the
iterated commutator of
$X_1,\ldots, X_n$
\begin{equation}\label{commutatoreiterato}
[\cdot,X_\alpha]= [  \cdot,\underbrace{X_1,\dots,X_1}_{\textrm{$\alpha_1$
times}},\underbrace{X_2,\dots,X_2}_{\textrm{$\alpha_2$
times}},   \dots,\underbrace{X_n,\dots,X_n}_{\textrm{$\alpha_n$
times}}].
\end{equation}
Here, only elements $X_1,\ldots, X_n$ of the basis of $\mathfrak g$ are
involved.
We agree that  $[\cdot,X_{(0,\ldots,0)}]={\rm Id}$. 
The \emph{generalized  structure constants} $c_{i\al}^k\in\R$, with
$\a\in\mathcal I = \N^n$ and $i,k\in\Z$ such that
$i,k\leq n$,  are defined via the relation
\begin{equation}\label{GSC}
  [X_i,X_\a]= \sum_{k\leq n} c^k_{i\a} X_k.
\end{equation} 
The above sum  is always finite. In fact, letting 
$d(\alpha) = \sum_{j=1}^n \alpha_j  d(j)$, we have $c_{i\alpha}^k = 0 $ if
$d(k)\neq d(i)+d(\alpha)$.

For $\al\in\I$ and $i\in\Z$ with $i\leq n$, we define the linear mapping
$\phi_{i\al}\in\Hom( \Rn;\R)$  
\begin{equation}\label{LinM}
 \phi_{i\al}(v) = \frac{(-1)^{|\a|}}{\a!}  \sum_{k\leq n}
 c^k_{i\a}v_k,\quad v=(v_1,\ldots,v_n)\in \Rn.
\end{equation}
Above, we let $\alpha! = \alpha_1 ! \cdot\ldots\cdot \alpha_n!$.
Moreover, we agree that $v_k=0$ for $k\leq 0$.


\begin{defi}\label{Pollo}
For each $i\in\Z$ with $i\leq n$ and $v\in \Rn$, we call the polynomial
$P_i^v:\Rn\to\R$
\begin{equation}\label{gestalt}
P_i^v(x)= \sum_{\a\in \I} \phi_{i\alpha}(v) 
 x^\a,\quad x\in\Rn,
\end{equation}
\noindent
\emph{extremal polynomial}  
of the Lie algebra $\Prol(\mathfrak g)$  with
respect to the basis $\{ X_j\}_{j\leq n}$.
\end{defi}

For any finite set of multi-indexes $\mathcal A\subset \mathcal I  $, consider
the polynomial 
\[
 \displaystyle P(x) = \sum_{\alpha\in\mathcal A } c_\a
x^\alpha,
\]
where  $c_\alpha\in\R $ for all $\alpha\in\A$. The
homogeneous degree
of the polynomial is $d(P) = \max\big\{ d(\alpha) : \alpha \in\mathcal A
\textrm{ such that }c_\alpha \neq0\big\}$.
We say that the  polynomial is homogeneous of degree $k\geq 0$ if $d(\alpha) =
k $ for all $\alpha \in\mathcal A$ such that $c_\alpha\neq0$.

Extremal polynomials  $P_i^v$ are indeed
polynomials. 
In fact, if $d(i)+ d(\alpha)\neq d(k) $ then we 
have $c_{i\alpha}^k =0$. Moreover, we have  $d(k) \leq d(n) = s$ and therefore
the sum
in \eqref{gestalt} ranges over the $\alpha\in\mathcal I$ such that $d(\alpha)
\leq s- d(i)$. So we have $d(P^v_i)\leq s -d(i)$.

The following theorem is the main result of the paper.

\begin{teo}\label{propderivpolinomi}
For any $v\in \Rn$, $i=1,\dots,n$, and $j\in\Z$ with $j\leq n$ there holds 
\begin{equation}\label{derivatepolinomi}
X_i P_j^v = \sum_{k\leq n} c_{ij}^k P_k^v.
\end{equation}
Moreover, the polynomials $\{P_j^v\}_{j\leq n}$ are uniquely determined by
\eqref{derivatepolinomi} 
for $i=1,\ldots,r$
and $P_j^v(0)=v_j$ for $j\leq n$.

\end{teo}

The identity $P^v_j(0)=v_j$ is proved in \eqref{linux} below.
The uniqueness  follows from this observation:  if $f$
is a smooth function on $G=\Rn$ such that $X_i f = 0 $
for all $i=1,\ldots,r$, then $f$ is   constant. 
Now, let  $\{Q_j^v\}_{j\leq n}$ be a family of polynomials satisfying
\eqref{derivatepolinomi} for $i=1,\ldots,r$, and $Q_j^v(0)=v_j$ for $j\leq n$.
Then, for $d(j) = s$ we have $X_i
P_j^v =0=X_iQ^v_j$ and thus $P_j^v -Q_j^v = P_j^v(0) - Q_j^v(0)=0$.
Assume by induction that $P_j^v =Q_j^v$ for $d(j)\geq \ell+1$
and take $j\leq n$ such that $d(j)=\ell$.
From \eqref{derivatepolinomi}, we have
\[
X_i P_j^v = \sum_{k\leq n} c_{ij}^k P_k^v =  \sum_{k\leq n} c_{ij}^k Q_k^v=
X_i Q_j^v,\quad i=1,\ldots,r,
\]
and thus $P_j^v - Q_j^v = P_j^v (0)- Q_j^v(0) =0 $.

Before starting the proof of  \eqref{derivatepolinomi}, we need three
preliminary lemmas. For  $i=1,\ldots,n$,   let $G_i\subset G = \Rn$ be the
subgroup of $G$ 
\[
\begin{split}
 G_ i & =\big\{ x\in G=\Rn: x_1=\ldots=x_{i-1} =0\big\}
      \\&
     =\big\{ \ex^{x_i X_i}\circ\cdots\circ\ex^{x_n X_n}(0) \in G=\Rn:
x_i,\ldots,x_n\in\R\big\}.
\end{split}
\]
We clearly have  $G_1 = G$. The following lemma is well-known and we   omit
the proof.

\begin{lemma} \label{L1}
Let $X_1,\ldots,X_n$ be vector fields in $\Rn$ satisfying \eqref{ExpCord}.
Then:

 \begin{itemize}
  \item[(i)] For all $i=1,\ldots,n$ and for all $x\in G_i$ we have  
          $X_i(x) = \partial / \partial x _i$.

  \item[(ii)] There exist functions  $f_{i\ell}:\Rn\to\R$, $i,\ell
=1,\ldots,n$, such that  
\[
   X_i(x) =\frac{ \partial}{\partial x _ i}  + \sum_{\ell:d(\ell)>d(i)} f_{i
\ell}(x) \frac{\partial}{\partial x_\ell},\quad x\in G=\Rn.
\]
 \end{itemize}
\end{lemma}

In the following two lemmas, we prove   formula \eqref{derivatepolinomi} in
two simplified situations that will serve as base for a long induction argument.

\begin{lemma}\label{L2}
For any  $i=1,\ldots,n$, $j\in\Z$ with $j\leq n$, and $v\in\Rn$  there holds
\begin{equation}\label{base1}
  X_i P_j^v   = \sum_{k\leq n} c_{ij}^ k P_k^v \quad \textrm{on } G_i.
\end{equation}
\end{lemma}

\begin{proof}
 For    $x\in G_i$, we have 
\begin{equation} 
\label{pinco}
 P_j^v(x) = \sum_{\ell \leq n} \sum_{   \alpha \in\I_i  } c_\a
c^\ell_{j\alpha}
v_\ell x^\a, 
\end{equation}
where we let $\I_i = \big\{ \alpha\in \I : \a_1=\ldots=\alpha_{i-1}=0
\big\} $ and
\[
 c_\alpha = \frac{(-1)^{|\alpha|} }{\alpha!}.
\]
By (i) in Lemma \ref{L1}, on $G_i$ we have
$X_i= \partial / \partial x_i$. Then, differentiating \eqref{pinco} we obtain
\begin{equation}
 \label{foro1}
    X_i P_j^v(x) 
             = - \sum_{\ell\leq n} \sum_{\beta \in\I_i}
c_\beta c^\ell_{j,\beta+\ex_i} v_\ell x^\beta,\quad x\in G_i.
\end{equation}

On the other hand,  for  $x\in G_i$ and $v\in \Rn$ we have
\begin{equation}\label{foro2}
 \sum_{k \leq n}  c_{ij}^ k  P_k^v(x) = \sum_{\ell,k\leq n}  \sum_{\beta\in
\I_i}
c_\beta c_{ij}^k c^\ell_{k\beta} v_\ell x^\beta.
\end{equation}

By \eqref{foro1} and \eqref{foro2}, proving the claim \eqref{base1} is
equivalent to show that for
all $\ell \leq n$ and for all $\beta\in\I_i$
we have  
\begin{equation}\label{cirox}
 -   c^\ell_{j,\beta+\ex_i}   =  \sum_{k \leq n}  c_{ij}^k c^\ell_{k\beta}.
\end{equation}

As $\beta\in \I_i$, there holds 
 $-[X_j, X_{\beta+\ex_i}]= 
[[X_i,X_j],X_\beta]$, and thus, by \eqref{GSC} and \eqref{StrCon2},
we obtain 
\begin{equation}\label{conc}
-    \sum_{\ell \leq n} c^\ell_{j,\beta+\ex_i} X_\ell  = 
 -[X_j, X_{\beta+\ex_i}]= 
[[X_i,X_j],X_\beta]
= \sum_{ k \leq n} c_{ij}^k  [X_k,X_\beta]
 =
\sum_{\ell,k \leq n} c_{ij}^k c^\ell_{k\beta}X_\ell
.
\end{equation}
Since $\{X_\ell\}_{\ell\leq n}$ is a basis, this proves \eqref{cirox}.
\end{proof}

\begin{lemma}\label{4} For all $j\in\Z$ with $j\leq n$,   $i=1,\ldots,n$ such
that $d(i)=s$,  
and $v\in \Rn$ there holds
\begin{equation}\label{depoli2}
X_i P_j^v = \sum_{k\leq n} c_{ij}^k P_k^v.
\end{equation}
\end{lemma}

\begin{proof}
 When $j=1,\ldots,n$, the left and right hand sides of \eqref{depoli2}
are   both identically $0$.   Recall that we have $d(P_j^v)\leq s-
d(j)<s$ and thus   $X_i P_j^v=0$.
We prove the claim \eqref{depoli2}
for any $j\leq n$, in particular for $j\leq 0$.
As $d(i) = s$, by (ii) in Lemma \ref{L1} we have $X_i = \partial /\partial
x_i$ on $G=\Rn$.
Then we have formula \eqref{foro1} with $\beta\in\I$ replacing $\beta \in \I_i$.
We  also have formula \eqref{foro2}, again with $\beta\in\I$ replacing $\beta
\in \I_i$. Then we are reduced to check identity \eqref{cirox} for all $\beta
\in\I$.

We claim that, for any $\beta \in \mathcal I$, we have
\begin{equation} \label{stern}
 [X_j, X_{\beta+\ex_i}]= -
[[X_i,X_j],X_\beta].
\end{equation}
We prove \eqref{stern} by induction on the length of $\beta$ and we assume
without loss of generality
that $i=n$. Notice that   the vector field $X_n$ commutes with any vector field
$X_k$, $k=1,\ldots,n$.

Assume that $\beta_k \neq0$ and $\beta _h = 0$ for   $h>k$. Then,
by the Jacobi identity, by $[X_k,X_n]=0$, 
and by  formula \eqref{stern} 
for $\beta - \ex_k$ replacing $\beta$ (and with $i=n$),
we have
\[
 \begin{split}
   [X_j, X_{\beta +\ex_n}] & =  [[ X_j, X_{\beta - \ex_k}], X_k], X_n]
  \\
 & =   -[[X_n, [X_j, X_{\beta-\ex_k}]], X_k]
\\
 & =   [[X_j, X_{\beta-\ex_k+\ex_n}], X_k]
\\
 &=-[[[X_n,X_j],
X_{\beta-\ex_k}], X_k]
\\
 &=-[[X_n,X_j],
X_{\beta}],
 \end{split}
\]
and we are finished.
 
Now, using \eqref{stern} the proof can be concluded as in
\eqref{cirox}-\eqref{conc}.
\end{proof}

\medskip
\begin{proof}[Proof of Theorem \ref{propderivpolinomi}] 
The proof of \eqref{derivatepolinomi} is a triple
nested
induction. We fix the vector  $v\in \Rn$ and we let  $P_j^v=P_j$. The claim in
Theorem \ref{propderivpolinomi} reads
\[ 
  X_iP_j = \sum_{k\leq n}  c_{ij}^k
P_k\quad\textrm{on $G_\ell$ for all $i,\ell=1,\ldots,n$ and $j\leq n$.}
\]

The first induction 
is descending on   $i=1,\ldots,n$ and the base of induction is 
\[
   X_nP_j  =\sum_{k\leq n}  c_{nj}^kP_k\quad \textrm{on
$G$ for all $j\leq  n$}.
\]
This statement holds by  Lemma \ref{4}, because $d(n)=s$.

The first inductive assumption is the following  
\begin{equation}\label{8}
    X_h P_j =\sum_{k\leq n}  c_{hj}^kP_k\quad\textrm{on $G$ for all
$h>i$ and for all   $j \leq n$.}
\end{equation}
Our goal is to prove identity   \eqref{8} also for $h=i$. 
The proof of this claim is a descending induction on   $j\leq n$. 
The base of induction for   $j=n$ is the following
\[
    X_i P_n  =\sum_{k\leq n}  c_{in}^k P_k\qquad\textrm{on $G$.}
\]
This identity holds because we have $X_i P_n =0$ ($P_n$ is constant), and  
$c_{in}^k=0$ for any $i\geq 1$ and for all $k$.

The second inductive assumption is the following  
\begin{equation}\label{9}
   X_i P_h = \sum_{k\leq n}  c_{ih}^k P_k\qquad\textrm{on $G_\ell$ for all 
$h>j$
and for all   $\ell=1,\ldots,n$.}
\end{equation}
The goal is to prove identity  \eqref{9} also for $h=j$.
The proof is a descending induction on  $\ell=1,\ldots,i$.
The base of induction for   $\ell=i$ is the following
\[
   X_i P_j =\sum_{k\leq n} c_{ij}^k P_k\qquad\textrm{on $G_i$},
\]
that is satisfied  by  Lemma \ref{L2}.

The third inductive assumption is  
\begin{equation}
\label{10}
 X_i P_j =\sum_{k\leq n} c_{ij}^k P_k\qquad\textrm{on $G_{\ell+1}$}
\end{equation}
for $\ell+1\leq i$. Our goal is to prove  identity \eqref{10} on
$G_\ell$.

In the following, we can use the three inductive assumptions  
\eqref{8}-\eqref{9}-\eqref{10}.

We prove \eqref{10} at the   point  $x=(0,\ldots,0,x_\ell,\ldots, x_n)\in
G_\ell$. Notice that, by Lemma \ref{L1},  we have  $X_\ell =
\partial /\partial x_\ell$  on $G_\ell$.
For any  $s\in\R$, we let  
\[
\begin{split}
   y_s  & = (0,\ldots,0,s,x_{\ell+1},\ldots,x_n)\in G_{\ell}
\\
&
 = \ex^{s X_\ell}(y_0) 
\\
&
= \ex^{(s-x_\ell) X_\ell}(x),
\end{split}
\]
where $y_0   = (0,\ldots,0,0,x_{\ell+1},\ldots,x_n)\in G_{\ell+1}$.
In particular, we have  $y_{x_\ell} = x$.
Integrating, we obtain the formula 
\begin{equation} \label{star}
\begin{split}
   P_j(x) & = P_j(y_0)+\int_0^{x_\ell}(X_\ell P_j)(y_s) ds
\\
& =  P_j(\ex^{-x_\ell X_\ell}(x)) +\int_0^{x_\ell}(X_\ell
P_j)(\ex^{(s-x_\ell)X_\ell}(x)) ds.
\end{split}
\end{equation}
 
 From the structure (ii) in
Lemma \ref{L1}  for the vector field $X_i$ with $i>\ell$,
we deduce that  for any $z=(0,\ldots,0,z_\ell,\ldots,z_n)\in G_\ell$ and
for any  
$t\in\R$ we have  
\[
\ex^{t X_i}(z) = (0,\ldots,0,z_\ell,*,\ldots,*)\in G_\ell,
\]
and therefore  $\ex^{t X_i}(x) = (0,\ldots,0,x_\ell,*,\ldots,*)$ and
$\ex^{t X_i}(y_s) = (0,\ldots,0,s,*,\ldots,*)$.
Replacing $x$ with 
$\ex^{t X_i}(x)\in G_\ell$ into the identity 
\eqref{star}, we get:
\begin{equation}\label{piox}
\begin{split}
   P_j(\ex^{t X_i}(x)) 
 & = P_j(\ex^{-x_\ell X_\ell}\circ \ex^{t X_i}(x))
+\int_0^{x_\ell}(X_\ell P_j)(\ex^{(s-x_\ell)X_\ell}\circ
\ex^{t X_i}(x)) ds
\\
  & = P_j(\ex^{-x_\ell X_\ell}\circ \ex^{t X_i}
 \circ \ex^{x_\ell X_\ell}  (y_0))
+\int_0^{x_\ell}(X_\ell P_j)(\ex^{(s-x_\ell)X_\ell}\circ
\ex^{t X_i}\circ \ex^{-(s-x_\ell)X_\ell} (y_s)) ds
\\
  & = P_j(y_0\cdot \exp(x_\ell X_\ell)\cdot  \exp(t X_i)\cdot \exp (-x_\ell
X_\ell))
\\
&\quad 
+\int_0^{x_\ell}(X_\ell P_j)(y_s\cdot 
\exp(-(s-x_\ell)X_\ell)\cdot \exp(t X_i)\cdot 
\exp((s-x_\ell)X_\ell ) ds.
\end{split}
\end{equation}

For any $z\in G=\Rn$, we define the left translation $L_z:G\to G$ and
the conjugate mapping
$\Gamma_z:G\to G$
\[
 L_z{y}=z\cdot y\quad \textrm{and} \quad
 \Gamma_z(y) = z\cdot y \cdot z^{-1}, \quad y\in G.
\]
With this notation, identity \eqref{piox} reads
\[
\begin{split}
   P_j(\ex^{t X_i}(x)) 
 & =  P_j(L_{y_0} \Gamma _{ \exp( x_\ell X_\ell)} (\exp({t X_i})))
   +\int_0^{x_\ell}(X_\ell P_j)(L_{y_s} \Gamma_
{\exp(-(s-x_\ell)X_\ell)}  (\exp (tX_i)) )  ds,
\end{split}
\]
and thus 
\[
   X_i P_j (x) = \lim_{t\to0} 
 \frac{P_j(\ex^{t X_i}(x)) - P_j(x)}{t}= A_{ij}(x)+ B_{ij}(x), 
\]
where
\[
\begin{split} 
  A_{ij}(x) & = \lim_{t\to0}
 \frac{1}{t} \big[ 
  P_j(L_{y_0} \Gamma_{\exp(x_\ell X_\ell)} 
     (\exp(t X_i))) - P_j(y_0) \big],\quad \textrm{and}
\\
B_{ij}(x) &= 
\lim_{t\to0} \frac{1}{t} \int_0^{x_\ell} 
 \big[
    X_\ell P_j
  ( L_{y_s} \Gamma _{\exp (-(s-x_\ell) X_\ell)} 
  (\exp( t X_i)))
  - X_\ell P_j(y_s)
\big]ds.
\end{split}
\]

In the following we denote by ${\rm Ad}_z$ the differential of
$\Gamma_z$. Recall that one has the formula  ${\rm Ad}_{\exp
(Z)}=\ex^{{\rm ad} (Z)}$, where ${\rm ad} (Z) (Y) = [Z,Y]$, see
\cite[Proposition 1.91]{knapp}.
We have:
%
\[
\begin{split}
   A_{ij}(x)  & 
%
=
[(dL_{y_0})(\textrm{Ad}_{\exp{x_\ell X_\ell}}(X_i)) ] P_j(y_0)
\\
&
=
\big\{
   X_i +  x_\ell [X_\ell,X_i] +\frac 12 x_\ell^2
 [X_\ell,[X_\ell,X_i]]+\ldots \big\} P_j(y_0)
\\
&
= \big\{
   X_i -  x_\ell [X_i,X_\ell] 
    +\frac 12 x_\ell^2
 [ [X_i,X_\ell], X_\ell]
 -\frac{1}{3!} x_\ell^3 [[[X_i,X_\ell],X_\ell],X_\ell]
+\ldots \big\} P_j(y_0)
\\
&
= \big\{
   X_i +\sum_{p=1}^s \frac{(-1)^p}{p!} x_\ell^p
   [X_i,X_{p \ex_\ell}]\big\}   P_j(y_0)
\\
&
= \big\{
   X_i +\sum_{p=1}^s \sum_{h=1}^n \frac{(-1)^p}{p!} x_\ell^p c^h_{i,p\ex_\ell}
X_h \big\}   P_j(y_0).
\end{split}
\]

Notice that for any $p\geq 1$, we have  $c^h_{i,p\ex_\ell} =0 $ for all $h\leq
i$. For $h>i$ we  can use the inductive assumption 
\eqref{8} and find 
\begin{equation}\label{primo}
    X_hP_j(y_0) = \sum_{k\leq n} c^k_{hj} P_k(y_0).
\end{equation}
Moreover, since $y_0\in G_{\ell+1}$, we can also use the inductive assumption 
\eqref{10} and find 
\begin{equation}\label{secondo}
  X_i P_j(y_0) = \sum_{k\leq n} c^k_{ij} P_k(y_0).
\end{equation}

 By \eqref{primo} and \eqref{secondo}, we obtain the following expression for
$A_{ij}$:
\[
\begin{split}
 A_{ij}(x)  &= 
   \sum_{k\leq n} \Big\{
   c_{ij}^k + \sum_{p=1}^s \sum_{h=1}^n
   \frac{(-1)^p}{p!} x_\ell^p  c_{hj}^k c_{i,p\ex_\ell}^h
\Big\} P_k^v(y_0)
\\
&
= \sum_{k\leq n} \Big\{ 
   c_{ij}^k + \sum_{p=1}^s \Big( 
   \frac{(-1)^p}{p!} x_\ell^p \sum_{h=1}^n c_{hj}^kc_{i,p\ex_\ell}^h
\Big)\Big\} 
\Big[
   \sum_{\a\in\I} \sum_{m\leq n} \frac{(-1)^{|\alpha|}}{\alpha!} c_{k\alpha}^m
v_m
y_0^\alpha
\Big]. 
\end{split}
\]

A similar computation for $B_{ij}(x)$ shows that
%
%
\[
\begin{split}
B_{ij}(x) & = 
\int_0^{x_\ell}
 \lim_{t\to0} \frac 1 t \big[
  X_\ell P_j  (L_{y_s}\Gamma_{\exp(-(s-x_\ell)X_\ell)} (\exp(t X_i))) - X_\ell
P_j (y_s) \big] ds
\\
&
=
\int_0^{x_\ell} 
[(dL_{y_s})(\textrm{Ad}_{(x_\ell-s) X_\ell}(X_i)) ] (X_\ell  P_j) (y_s) ds
\\
&
=
\int_0^{x_\ell} \Big\{ X_i +\sum_{p=1}^s \sum_{u=1}^n
\frac{(-1)^p}{p!}(x_\ell-s)^p c_{i,p\ex_\ell}^u X_u\Big\} (X_\ell P_j) (y_s) ds.
\end{split}
\]

By \eqref{base1}, on   $G_\ell$ we have the identity 
\[
  X_\ell P_j = \sum_{h\leq n}  c_{\ell j}^h P_h,
\]
and since $y_s\in G_\ell$  for all $s$, we obtain
\[
B_{ij}(x)  = 
\int_0^{x_\ell}
  \sum_{h\leq n}  c_{\ell j}^h \Big\{ X_i +\sum_{p=1}^s \sum_{u=1}^n
\frac{(-1)^p}{p!}(x_\ell-s)^p c_{i,p\ex_\ell}^u X_u\Big\} P_h (y_s)
ds.
\]

Notice that $c_{\ell j}^h =0$ for $h\leq j$ and  $c_{i,p\ex_\ell}^u =0$ for
$p\geq 1$ and $ u\leq i$.
For $h>j$ and $u>i$ we can use the inductive assumptions  \eqref{8} and
\eqref{9}. We obtain the following expression for $B_{ij}$:
\[
 \begin{split}
 B_{ij}(x) & =  \int_0^{x_\ell} \sum_{h,k\leq n}  c_{\ell j}^h \Big\{
     c_{ih}^k   + \sum_{p=1}^s \sum_{u=1}^n
\frac{(-1)^p}{p!} (x_\ell-s)^p c_{i,p\ex_\ell}^u c_{uh}^k  \Big\}
P_k(y_s)ds 
\\
&=\int_0^{x_\ell}   \Big[
   \sum_{h,k\leq n}   \sum_{u=1}^n \sum_{p=0}^s  c_{\ell j}^h
\frac{(-1)^p}{p!}(x_\ell-s)^p
c_{i,p\ex_\ell}^u c_{uh}^k\Big]
\Big[
   \sum_{\a\in\mathcal I}\sum_{m\leq n}  v_m \frac{(-1)^{|\alpha|} }{\alpha!}
c_{k,\alpha}^m y_s^\alpha\Big] ds. 
 \end{split}
\]

Using the integrals 
\[
\begin{split}
 \int_0^{x_\ell} (x_\ell-s)^p y_s^\alpha ds & 
 = \int_0^{x_\ell} (x_\ell-s)^p  s^{\alpha_\ell} x_1^{\alpha_1}\cdots \widehat{
x_\ell^{\alpha_\ell}} \cdots x_n^{\alpha_n} \,   ds 
\\
&
= \frac{p! \alpha_\ell!}{(\alpha_\ell+p+1)!} x^{\alpha +(p+1)
\ex_\ell},
\end{split}
\]
we finally have 
\[
\begin{split}
   B_{ij}(x)  &  =\sum_{h,k,m\leq n}  \sum_{u=1}^n  \sum_{p=0}^s
\sum_{\alpha\in\I_\ell} \frac{(-1)^{p+|\alpha|}} {p!\alpha!}
c_{\ell j}^h 
c_{i,p\ex_\ell}^u 
c_{uh}^k  c_{k\alpha}^m\frac{p!\alpha_\ell!}
{(\alpha_\ell+p+1)!} v_m x^{\alpha+(p+1)\ex_\ell}
\\
& 
=-\sum_{h,k,m\leq n}  \sum_{u=1}^n  \sum_{p=0}^s \sum_{\alpha\in\I_\ell}
\frac{(-1)^{|\alpha+(p+1)\ex_{\ell}|}} {(\alpha+(p+1)\ex_\ell)!}
c_{\ell j}^h 
c_{i,p\ex_\ell}^u 
c_{uh}^k  c_{k\alpha}^m
 v_m x^{\alpha+(p+1)\ex_\ell}.
\end{split}
\]

Analogously, we have:
\[
  A_{ij}(x)  = \sum_{h,k,m\leq n}\sum_{p=0}^s\sum_{\alpha\in\I_{\ell+1}}
 \frac{(-1)^p}{p!} c_{i,p\ex_\ell}^h c_{hj}^k x_\ell^p
\frac{(-1)^{|\alpha|}} {\alpha!} c_{k\alpha}^m v_m y_0^\alpha,
\]
where the sum ranges over  $\alpha\in \I_{ \ell+1}$, because  
$y_0 = (0,\ldots,0, x_{\ell+1},\ldots, x_n)$. For $\alpha\in \I_{ \ell+1}$,
we have 
\[
    \frac{(-1)^p}{p!}  \frac{(-1)^{|\alpha|}}{\alpha!}
=   \frac{(-1)^{p+|\alpha|}}{(\alpha+p\ex_\ell)!}  
=   c_{\alpha+ p \ex_\ell},
\]
and, moreover,  $x_\ell^p y_0^\alpha = x^{\alpha + p\ex_\ell}$.
Thus,
\[
  A _{ij}(x) = \sum_{k,h,m\leq n} \sum_{p=0}^s\sum_{\alpha\in\I_{\ell+1}}
  c_{\alpha+p\ex_\ell} 
  c_{i,p\ex_\ell}^h 
  c_{hj}^k  
  c_{k\alpha}^m v_m x^{\alpha+p\ex_\ell}.
\]

Our goal is to prove that the quantity  $X_i P_j(x) = A_{ij}(x)+B_{ij}(x)$,
$x\in G_\ell$,  equals  
\[
  C_{ij}(x) =\sum_{k\leq n}  c_{ij}^k P_k (x)
   = \sum_{k,m\leq n} \sum_{\beta \in I_\ell}
      c_\beta c_{ij}^k c_{k\beta}^m v_m x^\beta.
\]

The coefficient of 
$c_\beta  v_m x^\beta$ in $C_{ij}(x)$ is 
\[
   \chi_{C}^{m,\beta} = \sum_{k\leq n} c_{ij}^k c_{k\beta}^m.
\]

The equation $\alpha + p \ex_{\ell} = \beta$ with
$\alpha_1=\ldots=\alpha_\ell=0$
implies
$p=\beta_\ell$
and $ \alpha =\beta - p \ex_\ell = \beta -\beta_\ell \ex_\ell$.
Thus, the coefficient of $c_\beta  v_m x^\beta$ in $A_{ij}(x)$ is
\[
  \chi_{A}^{m,\beta} = \sum_{k,h\leq n} c_{i,\beta_\ell\ex_\ell}^h c_{hj}^k
c_{k,\beta-\beta_\ell\ex_\ell}^m.
\]

We compute the coefficient $\chi_B^{m,\beta}$ of  $c_\beta  v_m x^\beta$ in
$B_{ij}(x)$.
The equation  $\alpha + (p+1)
\ex_\ell = \beta$ implies the following:  when 
$\beta_\ell=0$ the equation has no solution
and thus  $\chi_B^{m,\beta}=0$; when $\beta_\ell\geq
1$ we have $0\leq p\leq \beta_\ell-1$ and $
   \alpha = \beta - (p+1) \ex_\ell$, and therefore 
\[
   \chi_B^{m,\beta} = - \sum_{h,k\leq n}
\sum_{ u=1}^n \sum_{p=0}^{\beta_\ell-1} c_{\ell j}^h
    c_{i,p\ex_\ell} ^u
    c_{uh}^k
    c_{k,\beta -(p+1) \ex_\ell}^m,\quad \beta_\ell\geq 1.
\]

When  $\beta_\ell=0$, the identity $\chi_A^{m,\beta} +\chi_B ^{m,\beta}
=
\chi_C ^{m,\beta} $ is easily verified.

When  $\beta_\ell\geq 1$, 
the identity $\chi_A^{m,\beta} +\chi_B ^{m,\beta}
=
\chi_C ^{m,\beta} $ is proved as soon as we verify the following 
\begin{equation}\label{lop}
\sum_{k\leq n}  c_{ij}^k c_{k\beta}^m =
\sum_{k,h\leq n} c_{i,\beta_\ell\ex_\ell}^k  c_{hj}^k
c_{k,\beta-\beta_\ell\ex_\ell}^m
-
\sum_{h,k\leq n} \sum_{ u=1}^n \sum_{p=0}^{\beta_\ell-1} c_{\ell j}^h 
    c_{i,p\ex_\ell} ^u
    c_{uh}^k
    c_{k,\beta -(p+1) \ex_\ell}^m .
\end{equation}
Contracting with  $X_m$ the left hand side of \eqref{lop} we find 
\begin{equation}\label{bia}
\begin{split}
  \sum_{m,k\leq n}   c_{ij}^k c_{k\beta}^m X_m & =
   \sum_{k\leq n}  c_{ij}^k[X_k, X_\beta]     
   = [[X_i,X_j],X_\beta]
\\&
   = [[[X_i,X_j],X_{\beta_\ell\ex_\ell}], X_{\beta-\beta_\ell\ex_\ell}].
\end{split}
\end{equation}
Contracting with $X_m$ the first term in the right hand side of \eqref{lop}
we
find 
\begin{equation}\label{bio}
\begin{split}
\sum_{k,h,m\leq n} c_{i,\beta_\ell\ex_\ell}^h  c_{hj}^k
c_{k,\beta-\beta_\ell\ex_\ell}^m X_m
&    = \sum_{k,h\leq n} c_{i,\beta_\ell\ex_\ell}^h  c_{hj}^k
[X_k,X_{\beta-\beta_\ell\ex_\ell}]
\\
&
   =  \sum_{h=1}^n c_{i,\beta_\ell\ex_\ell}^h  
[[X_h,X_j],X_{\beta-\beta_\ell\ex_\ell}]
\\&
   = [[[X_i,X_{\beta_\ell\ex_\ell} ],X_j],X_{\beta-\beta_\ell\ex_\ell}].
\end{split}
\end{equation}
Contracting the  second term in the right hand side of \eqref{lop} we find
\begin{equation}\label{biu}
-
\sum_{h,k,m\leq n} \sum_{ u=1}^n \sum_{p=0}^{\beta_\ell-1} c_{\ell j}^h 
    c_{i,p\ex_\ell} ^u
    c_{uh}^k
    c_{k,\beta -(p+1) \ex_\ell}^m X_m
= \sum_{p=0}^{\beta_\ell-1}
[[[[X_i,X_{p\ex_\ell}],[X_j,X_\ell]],X_{(\beta_\ell-p-1)\ex_\ell}],X_{
\beta-\beta_\ell \ex_\ell}].
\end{equation}
We used $\beta_1=\ldots=\beta_\ell=0$.

In \eqref{bia}, \eqref{bio}, and \eqref{biu}, commutators have the same tail
$X_{\beta - \beta_\ell \ex_\ell}$. Thus the claim \eqref{lop} follows provided
that we prove the following identity:
\begin{equation}\label{WER}
[[X_i,X_j],X_{\beta_\ell\ex_\ell}]
=
[[X_i,X_{\beta_\ell\ex_\ell} ],X_j]
+\sum_{p=0}^{\beta_\ell-1}
[[[X_i,X_{p\ex_\ell}],[X_j,X_\ell]],X_{(\beta_\ell-p-1)\ex_\ell}].
\end{equation} 

We prove \eqref{WER} by induction on  $\beta_\ell\geq 1$.
The base of induction for  $\beta_\ell=1$ is  
\[
[[X_i,X_j],X_\ell] =
[[X_i,X_\ell],X_j]
+ [X_i,[X_j,X_\ell]],
\]
and this holds by the Jacobi identity.

By induction, we assume that  \eqref{WER} holds for $\beta_\ell$ and we prove
it  for $\beta_\ell+1$.
We have:
\[
\begin{split}
[[X_i,X_j],X_{(\beta_{\ell+1} )\ex_\ell}]
& = [[[X_i,X_j], X_{\beta_\ell\ex_\ell}], X_\ell]
\\&
=[[[X_i,X_{\beta_\ell\ex_\ell} ],X_j],X_\ell]
+\sum_{p=0}^{\beta_\ell-1}
[[[[X_i,X_{p\ex_\ell}],[X_j,X_\ell]],X_{(\beta_\ell-p-1)\ex_\ell}],X_\ell]
\\
&
= [[X_i,X_{(\beta_\ell+1)\ex_\ell} ],X_j]
+ [[X_i,X_{\beta_\ell\ex_\ell}],[X_j,X_\ell]]
\\
&\qquad
+\sum_{p=0}^{\beta_\ell-1}
[[[X_i,X_{p\ex_\ell}],[X_j,X_\ell]],X_{(\beta_\ell-p)\ex_\ell}]
\\
&
=
[[X_i,X_{(\beta_\ell+1)\ex_\ell} ],X_j]
+\sum_{p=0}^{\beta_\ell}
[[[X_i,X_{p\ex_\ell}],[X_j,X_\ell]],X_{(\beta_\ell-p)\ex_\ell}].
\end{split}
\] 
This finishes the proof of   Theorem \ref{propderivpolinomi}.
\end{proof}

In the next proposition, we list some elementary properties of extremal
polynomials that are used in Section \ref{section:alg-abn-extr}.

\begin{prop}\label{lemma-0}
Extremal polynomials have the following properties. 

{\upshape (i)}  For all  $i=1,\ldots,n $ and $v\in\Rn$, we have
$P_i^v(0)=v_i$. 

{\upshape (ii)} For all  $i\in\Z$ with $i\leq 0 $ and $v\in\Rn$, we have
$P_i^v(0)=0$. 

{\upshape (iii)} If $v\in\R^n$ is such that $P_i^{v}=0$ for all
$i=1,\dots,n$ with $d(i)=1$, then $v=0$.
\end{prop}

\begin{proof}
(i) For any $i=1,\ldots,n$, we have
\begin{equation}\label{linux}
   P_i^v(0) =  \phi_{i0}(v) =  \sum_{k\leq n} c_{i0}^k v_k = v_i, 
\end{equation}
because $c_{i0}^k = \delta_{ik}$, the Kronecker symbol.

(ii) This follows   as in (i) from the agreement that $v_k=0$ for $k\leq 0$.

(iii) Assume that $P^v_k=0$ when $d(k)=1$. We claim that $P^v_k=0$ for 
 all $k=1,\ldots, n$. Because of (i), this will imply $v=0$.
The proof is by induction on $d(k)$.
Assume that $P^v_k=0$ for all $k=1,\ldots,n$ such that
$d(k) \leq \ell$, where $\ell<s$.
Take $k=1,\ldots, n$ such that $d(k) = \ell+1$.
By the stratification, there are constants $d_{ij}^k\in\R$,
$d(i)=1$ and $d(j) = \ell$, such that
\[
    X_k = \sum_{\substack { d(i)=1 \\ d(j) = \ell}} d_{ij}^k [X_i,X_j]
    = \sum_{\substack { d(i)=1 \\ d(j) = \ell}}  \sum_{\substack
{ h\leq n }}  
   d_{ij}^k c_{ij}^h X_h,
\]
and thus we have the identity
\begin{equation}
 \label{stirox}
   \sum_{\substack {d(i)=1 \\  d(j) = \ell}  }
   d _{ij} ^k c_{ij}^h = \delta_{hk}.
\end{equation}
From $P_j ^v =0$  for $d(j)=\ell$, it follows that $X_i P_j^v=0$. Now, from
Theorem \ref{propderivpolinomi} and \eqref{stirox} we obtain
\[
   0 = \sum_{\substack { d(i) = 1 \\  d(j) = \ell}} d_{ij}^k X_i P_j^v 
    = \sum_{\substack {d(i) = 1\\ d(j) = \ell}} \sum_{h\leq n} 
d_{ij}^k c_{ij}^h P_h^v= P_k^v.     
\]
This finishes the proof.
\end{proof}

\section{Algebraic characterization of abnormal extremals}
\label{section:alg-abn-extr}

In this section, we prove the   theorems about the algebraic characterization
of  abnormal extremals. The Lie
group $G$ is identified with $\Rn$ via exponential coordinates 
as in \eqref{ExpCord}, where $X_1,\ldots,X_n$ is a basis of
$\mathfrak g = \mathrm{Lie}(G)$ adapted to the stratification. 
%
%
%
The vector fields $X_1,\ldots,X_r$ are a basis for $\mathfrak
g_1$ (and thus  generators for $\mathfrak g$).
The distribution of $r$-planes 
$\mathcal D(x)
= \mathrm{span}\{X_1(x) , \ldots,X_r(x) \}$,   $x\in G$, is called
{horizontal distribution} of $G$.
A Lipschitz curve
$\ga:[0,1]\to G=\Rn$ is $\mathcal D$-{horizontal}, or simply
 {horizontal}, if there exists
a vector of functions $h=
(h_1,\ldots,h_r)\in  
L^\infty([0,1];\R^r)$ such that
\[
\dot\ga =
 \sum_{j=1}^r h_j X_j(\ga),\quad \text{a.e. on }[0,1].
\]
The functions $h$ are called controls of $\ga$ and, when $\gamma:[0,1]\to\Rn$
is given by the coordinates  $\gamma=(\gamma_1,\ldots,\gamma_n)$, 
we have $h_j = \dot\gamma_j$, $j=1,\ldots,r$.
This follows from the structure of the vector fields $X_1,\ldots,X_r$ described
in Lemma \ref{L1}.

Let $g_x$ be the quadratic form on $\D(x)$ 
making $X_1,\ldots,X_r$ orthonormal.
The
horizontal length of a horizontal curve $\ga$ is   
\[
 L(\gamma) = \Big( \int_0^1 g_{\ga(t)}(\dot\gamma(t) )dt\Big)^{1/2} 
=\Big( \int_0^1|h(t)|^2 dt\Big)^{1/2} .
\]
For any couple of points $x,y\in G$, we can define the function 
\begin{equation}\label{dist}
 d(x,y) = \inf\Big\{ L(\ga): \text{$\ga$ is  horizontal, $\ga(0)=x$
and $\ga(1)=y$} \Big\}.
\end{equation}
The above set is always nonempty  and $d$ is a distance on
$G$.
A Lipschitz curve $\ga$ providing the  minimum in \eqref{dist}
is called a \emph{length  minimizer}.


We denote by $\vartheta_1,\ldots,\vartheta_n$ the basis 
of $\mathfrak g^*$ dual to the basis $X_1,\ldots,X_n$ of $\mathfrak g$. A
Lipschitz curve of $1$-forms
$\lambda\in \mathrm{Lip}([0,1];\mathfrak g^*)$
is a curve  
 \[ \lambda  = \lambda_1 \vartheta_1+\ldots+\lambda_n
\vartheta_n,\] where $\lambda_i:[0,1]\to\R$, $i=1,\ldots,n$, are Lipschitz
functions.
We call $\lambda_1,\ldots,\lambda_n$ the coordinates of $\lambda$.

\begin{teo}\label{teo:eqdualvar1} Let $\ga$ be a length minimizer in $G$. 
Then there exist a number  $\lambda_0\in\{0,1\}$ and a Lipschitz curve
$\lambda \in \mathrm{Lip}([0,1];\mathfrak g^* )$ with coordinates
$\lambda_1,\ldots,\lambda_n$
such that:
\begin{itemize}
 \item[(i)] $\lambda_0+|\lambda|\neq 0$ on $[0,1]$;

 \item[(ii)] $\lambda_0 \dot\gamma_j +\lambda_j =0$ on
$[0,1]$ for all $j=1,\ldots, r$; 

 \item[(iii)]  for all    $i=1,\ldots,n$, we have
\begin{equation}\label{81}
\dot\lambda_i = - \sum_{k=1}^n \sum_{j=1}^r  c_{ij}^k  \dot \gamma_j 
\lambda_k\quad\text{a.e.~on }[0,1],
\end{equation}
where $c_{ij}^k$ are the structure constants \eqref{GSC}.
\end{itemize}
\end{teo}

\noindent 

Theorem \ref{teo:eqdualvar1} is a version of Pontryagin Maximum Principle
adapted to the present setting.
Equations \eqref{81} are called  \emph{adjoint equations}.
We refer to \cite[Chapter 12]{AS} for a proof of Theorem
\ref{teo:eqdualvar1} in a more general framework. 
The version \eqref{81} of the adjoint equations is derived in \cite{LLMV},
Theorem 2.6. The curve $\lambda$ is called a \emph{dual curve}
of $\gamma$.

\begin{defi}   
We say that a   horizontal curve
$\ga:[0,1]\to G$  is an \emph{extremal} if there exist
$\lambda_0\in\{0,1\}$   
and a curve of $1$-forms
$\lambda\in\mathrm{Lip}([0,1];\mathfrak g^*)$
such that i), ii), and iii) in Theorem \ref{teo:eqdualvar1} hold.

We say that $\ga$ is a \emph{normal extremal} if there exists such a pair
$(\lambda_0,\lambda)$  with $\lambda_0=1$.  

We say that $\ga$ is an \emph{abnormal extremal}
if there exists such a pair with $\lambda_0=0$. 

We say that $\ga$ is a
\emph{strictly abnormal extremal} if $\ga$ is an abnormal extremal but not a
normal one.
 \end{defi}

\begin{remark} If $\gamma$ is an abnormal extremal with dual curve $\lambda$,
then we have
\begin{equation}
   \label{condnecessabnormali}
   \lambda_1 = \ldots = \lambda_r = 0 \quad \text{on }[0,1].
\end{equation}
This follows from condition ii) of Theorem \ref{teo:eqdualvar1} with
$\lambda_0=0$.
\end{remark}

\begin{teo}\label{teointduale}
Let $G$ be a stratified nilpotent Lie group, let $\ga:[0,1]\to G$
be a horizontal curve such that
$\ga(0)=0$, and let $  \lambda_1,\ldots,\lambda_n:[0,1]\to\R$ be
Lipschitz functions. The following statements are
equivalent:

\begin{itemize}
 \item[(A)] The functions $\lambda_1,\ldots,\lambda_n$  solve a.e.~the system of
equations \eqref{81}.

 \item[(B)] There exists $v\in\Rn$  such that, for all $i=1,\ldots,n$,
we have
\begin{equation}\label{eqlaP}
\la_i(t) = P_i^{v}(\ga(t)),\quad t\in[0,1],
\end{equation}
and in fact $v=(\la_1(0),\ldots,\la_n(0))$.
\end{itemize}
\end{teo}

\begin{proof} (B)$\Rightarrow$(A) 
Let $\lambda_1,\ldots,\lambda_n$ be functions on $[0,1]$ defined as in
\eqref{eqlaP} for some $v\in\Rn$. Then, by Proposition \ref{lemma-0} we have $v
= (\lambda_1(0),\ldots,\lambda_n(0))$. Moreover, by Theorem
\ref{propderivpolinomi} we have, a.e.~on $[0,1]$ and for any $i=1,\ldots,n$,
\begin{equation}\label{stax}
    \dot\lambda_i = \frac{d}{dt} P^v_i(\gamma) 
  = \sum_{j=1}^r\dot\gamma_j X_j P^v_i(\gamma)
  = \sum_{j=1}^r \sum_{k=1}^n c_{ji}^k \dot \gamma_j  
      P^v_k(\gamma)
 =  - \sum_{j=1}^r \sum_{k=1}^n c_{ij}^k \dot \gamma_j  
      \lambda_k.
\end{equation}

(A)$\Rightarrow$(B) Let $v= (\lambda_1(0),\ldots,\lambda_n(0))\in\Rn$. Then the
functions $\lambda_i$ and the functions $\mu_i = P^v_i(\gamma)$,
$i=1,\ldots,n$, 
solve the system of differential equation \eqref{81} 
with the same initial condition. By the uniqueness of the solution we have
$\lambda_i = \mu_i$.
\end{proof}

Next, we define the notion of corank for an abnormal extremal.
For any
$h\in L^2([0,1];\R^r)$, let $\ga^h$ be the solution of the problem
\[
 \dot\ga^h = \sum_{j=1}^r h_j X_j(\ga^h),\quad \ga^h(0)=x_0.
\]
The mapping $\mathcal E: L^2([0,1];\R^r)\to  G=\Rn$, $\mathcal E (h) =
\ga^h(1)$, is called 
the \emph{end-point mapping}  with initial point $x_0\in G$.
A horizontal curve $\gamma$ starting from
$x_0$ with controls $h$ is an abnormal extremal if and only if
there exists $v\in\R^{n}$, $v\neq 0$,  
such that 
\begin{equation}\label{POX9}
    \langle d\mathcal E(h)w,v\rangle =0
\end{equation}
for all $w\in L^2([0,1];\R^r)$. Here, $d\mathcal E(h)$ is the differential of
$\mathcal E$ at the point $h$.

\begin{defi}\label{CORANK}
 The \emph{corank} of an abnormal extremal $\ga:[0,1]\to\Rn$ with controls $h$
is the
integer $\mathrm{corank}(\gamma)=n-\mathrm{dim}\big(
\mathrm{Im} \, d \mathcal{E}(h)\big) \geq 1$.
\end{defi}

 \noindent  The  {corank} of an abnormal extremal $\ga:[0,1]\to G$  
is the
integer $m\geq 1$ such that there exist precisely $m$
linearly independent dual curves
$\lambda^1,\ldots,\lambda^m\in\mathrm{Lip}([0,1];\mathfrak {g}^*)$ each  solving
the system
of adjoint equations \eqref{81}   with initial conditions
$\lambda^1(0),\ldots,\lambda^m(0)\in \Rn$ that are orthogonal to the image
of
the differential
of the end-point mapping.

\begin{defi} \label{AV} Let $G=\R^n$ be a stratified nilpotent Lie group
of rank $r$.
 For any $v\in\Rn$, $v\neq0$, we call the set   
\[
  Z_v = \{x\in  \Rn: P_j^v (x)= 0\textrm{ for all $j\in\Z$ with $j\leq r$} \}
\]
an \emph{abnormal variety} of  $G$ \emph{of corank $1$}.

For linearly
independent vectors $v_1,\ldots, v_m\in\Rn$, $m\geq 2$, we call the set
$Z_{v_1}\cap \ldots \cap Z_{v_m}$ an  \emph{abnormal variety} of  $G$
\emph{of corank $m$}.
\end{defi}

\begin{remark}
The abnormal variety $Z_v$ is the intersection 
of the zero sets of all the polynomials $P^v_j$ with $j\leq r$. In this
intersection, also indexes $j\leq 0$ are involved. With this respect, Definition
\ref{AV} differs from the analogous definition in \cite{LLMV}.
\end{remark}

\begin{teo}\label{cor:polin} 
Let $G=\Rn$ be a stratified nilpotent Lie group and let $\ga:[0,1]\to G$
be a  horizontal curve with
$\ga(0)=0$. The following statements are
equivalent:

\begin{itemize}
 \item[(A)] The curve $\ga$ is an abnormal extremal with
$\mathrm{corank}(\gamma)\geq m \geq 1$.

 \item[(B)] There exist  $m\geq 1$ linearly independent vectors
$v_1,\ldots,v_m\in\Rn$ such that $\gamma(t)\in  Z_{v_1}\cap\ldots\cap
Z_{v_m}$ for all $t\in[0,1]$.
\end{itemize}
\end{teo}

\begin{proof} 

(A)$\Rightarrow$(B)  Let  $\ga$ be an abnormal extremal and let 
$\la = \la_1\vartheta_1+\ldots + \la_n\vartheta_n$  be a Lipschitz curve of
$1$-forms with coordinates $\la_1,\ldots,\la_n$ solving the system of
differential equations \eqref{81} and such
that $\la\neq0$ pointwise. By Theorem \ref{teointduale}, there is a $v\in\Rn$,
$v\neq0$, such that $\la_i = P_i^v(\ga)$ for all $i=1,\ldots,n$. In
particular,   there holds $(\lambda_1(0),\ldots,\lambda_n(0)) = v$.
For abnormal
extremals, we have $\la_1=\ldots=\la_r=0$ on $[0,1]$, i.e., 
$\gamma(t)\in \big\{ x\in\Rn : P_j^v(x) = 0 \textrm { for all
$j=1,\ldots,r$}\big\} $
for all $t\in[0,1]$.  

We claim that $P_j^v(\gamma)=0$ also for all $j\in\Z$ with $j\leq 0$. The proof
is by induction on $d(j)$. The base of induction for $d(j)=1$ is proved above.
By induction, assume that $P_k^v(\gamma)=0$ for all $k\in\Z$ such that
$d(i)<d(k)\leq 1$. 
As in \eqref{stax}, by Theorem \ref{propderivpolinomi} we have
\[
 \frac{d}{dt} P^v_i(\gamma) 
  = -\sum_{j=1}^r \sum_{k\leq n} c_{ij}^k \dot \gamma_j  
      P^v_k(\gamma)=0,
\]
because $c_{ij}^k=0$ for $d(k)\leq d(i)$, for any $j=1,\ldots,r$. It follows
that $P_i^v(\gamma)=P_i^v(\gamma(0))$ is constant. By Proposition
\ref{lemma-0} and $\gamma(0)=0$, it is the zero constant. This shows that
$\gamma(t) \in Z_v$ for all $t\in[0,1]$.

From the above argument, we conclude  that if
$\mathrm{corank}(\gamma)\geq m$ then there exist 
at least $m$ linearly independent vectors $v_1,\ldots,v_m\in\Rn$ su that 
$\gamma(t) \in Z_{v_i}$ for all $t\in[0,1]$ and $i=1,\ldots,m$.

\medskip
(B)$\Rightarrow$(A) Let $v_1,\ldots,v_m\in\Rn$ be linearly independent vectors
such that (B) holds.
For any $v=v_1,\ldots,v_m$,  the curve $\la=\la_1\vartheta_1+ \ldots +
\la_n\vartheta_n$ with coordinates $\lambda_i = P^v_i(\gamma)$ satisfies the
system \eqref{81}, by Theorem
\ref{teointduale}. Moreover, we have $\la_1=\ldots=\la_r=0$ because $\gamma$ is
an abnormal extremal. From $v\neq 0$, it
follows that $(\la_1(0),\ldots,\la_n(0)) =v \neq 0$, and from the uniqueness of
the solution to
\eqref{81} with initial condition, it follows that $\la\neq0$ pointwise on
$[0,1]$. As $v_1,\ldots,v_m$ are linearly independent, then the corresponding
curves $\la$ are also linearly independent. We conclude
$\mathrm{corank}(\gamma)\geq m$.
\end{proof}

\begin{remark}\label{R2}
In the rank 2 case, an abnormal extremal $\gamma$   satisfies
$\lambda_1 = P_1^v(\gamma)=0$ and $\lambda_2 = P_2^v (\gamma)=0$. When $\gamma$
is not a constant curve, these equations imply that also $\lambda_3
=0$. Indeed,  the
equations \eqref{81} for $\lambda_{1}$ and $\lambda_{2}$ are 
\[
\dot \lambda_{1} = \dot \gamma_{2}\lambda_{3}\quad \textrm{and}\quad 
\dot \lambda_{2} = -\dot \gamma_{1}\lambda_{3}.
\]
Thus, if $\dot\gamma_1^2+\dot\gamma_2^2\neq0$, it must be $\lambda_3=0$.

\end{remark}

\section{The abnormal set in the free 
group of rank 2 and step 4}\label{esempi}

In this section, we present a method to compute a set containing
all abnormal extremals passing through one point. The method works when
each abnormal extremal is in the zero sets of a sufficiently large number
of extremal polynomials. This is the case when the prolongation is sufficiently
large. In order to make the presentation clear, we focus   on the free group
of rank 2 and step 4. The key point is to show that a certain polynomial is
nontrivial, see Theorem \ref{5.2} below.

Let $\mathfrak g = \mathfrak
g_1\oplus\cdots\oplus \mathfrak g_4$ be the free nilpotent Lie
algebra with rank
 $r=2$ and step $s=4$. This algebra is $8$-dimensional.  
Let $X_1,\ldots,X_8$ be a Hall basis of $\mathfrak g$ (see \cite{H} and
\cite{LLMV}). This basis satisfies the following structure relations
\begin{equation}
\label{pill}
\begin{split}
   X_3 & = [X_2,X_1],\quad
   X_4 = [X_3,X_1],\quad
   X_5 = [X_3,X_2],
\\
   X_6 & = [X_4,X_1],\quad
   X_7 = [X_4,X_2],\quad
   X_8 = [X_5,X_2].
\end{split}
\end{equation}
In the Grayson-Grossman model \cite{GrGr}, $X_1,\ldots,X_8$ can be identified with vector
fields in $\R^8$ where
\begin{equation}
 \label{CAMP}
\begin{split}
   X_1 & = \frac{\partial}{\partial x_1},\quad
\\
   X_2 & = \frac{\partial}{\partial x_2}
   -x_1\frac{\partial}{\partial x_3}
   +\frac {x_1^2}{2} \frac{\partial}{\partial x_4}
   +x_1x_2 \frac{\partial}{\partial x_5}
   -\frac{x_1^3}{6} \frac{\partial}{\partial x_6}
   -\frac{x_1^2 x_2}{2} \frac{\partial}{\partial x_7}
   -\frac{x_1 x_2^2}{2} \frac{\partial}{\partial x_8}.
\end{split}
\end{equation}

The Tanaka prolongation
$\mathrm{Prol}(\mathfrak g)$ is finite dimensional and by Theorem 1 in
\cite{BEN}, we have 
\[
 \mathrm{Prol}(\mathfrak g) = \mathfrak g_0 \oplus \mathfrak g, 
\]
where $\mathfrak g_0$ is the vector space of strata preserving derivations of
$\mathfrak g$. In fact, $\mathfrak g_0$ is isomorphic to $\Endo(\mathfrak g_1)$
and thus $\mathfrak g_0$ is $4$-dimensional, because $\mathrm{dim}(\mathfrak
g_1)=2$.
We extend the basis for $\mathfrak g$ to  a basis $X_{-3}, X_{-2},\ldots, X_8$
of $\mathrm{Prol}(\mathfrak g)$  and we denote by $c_{j\alpha}^k$, with
$j=-3,-2,\ldots,8$, $ k=1,\ldots,8$, and $\alpha\in\mathcal I =\N^8$, the
generalized
structure constants of $ \mathrm{Prol}(\mathfrak g)$.

Let $G$ be the free   nilpotent Lie group of rank $r=2$ and step $s=4$.
Via exponential coordinates of the second type, $G$ is diffeomorphic to $\R^8$.
The extremal polynomials in $\R^8$ associated with the Tanaka prolongation of
$\mathfrak g = \mathrm{Lie}(G)$ are, for any fixed $v\in\R^8$,
\begin{equation}\label{polli}
\begin{split}
  P_j^v (x) & 
 = \sum_{\alpha\in \mathcal I} \frac{(-1)^{|\alpha|}}{\alpha!} \sum_{k=1}^8
c_{j\alpha}^ k
v_ k x^\alpha= \sum_{k=1}^8  v_k   Q_{jk}(x)
,\quad x\in\R^8, 
\end{split}
\end{equation} 
where $Q_{jk}$ are the polynomials in $\R^8$
\begin{equation}\label{pipox}
 Q_{jk}(x)=  \sum_{\alpha\in \mathcal I} 
 \frac{(-1)^{|\alpha|}}{\alpha!}c_{j\alpha}^ k x^\alpha,\quad x\in\R^8.
\end{equation}
Above, we have $j=-3,-2,\ldots, 8$ and $k=1,\ldots,8$.
Notice that  $c_{j\alpha}^k= 0$ if   $d(j) + d(\alpha)\neq  d(k)$.
It follows that $Q_{jk}$ is a polynomial
with homogeneous degree $d(k)-d(j)$.  

Let $v\in\R^8$ be a vector such that $v_1=v_2=v_3=0$.
The first  three polynomials $P_{1}^v$,
$P_{2}^v$, and $P_3^v$  are the following
\begin{equation}
\label{51}
\begin{split}
P_{1}^v(x)  =  & 
v_4 x_3 
- v_5 \frac{x_2^{2}}{2} 
+v_6x_4 
+v_7x_5
+ v_8\frac{x_2^{3}}{6},
\\
P_{2}^v(x)  =   &  
v_4 \frac{x_1^{2}}{2} +(x_3 
+x_1 x_2) v_5 
-\frac{x_1^{3}}{6} v_{6}+\big(x_4
-\frac{x_1^{2}x_2}{2}\big)  v_7 +\big(x_5
-\frac{x_1x_2^{2}}{2}\big) v_8,
\\
P_{3}^v(x)  = & 
- v_4 x_1 - v_5 x_2
+v_6  \frac{x_1^{2}}{2}+ v_7 x_1 x_2 +v_8  \frac{x_2^{2}}{2}.
\end{split}
\end{equation}
These polynomials can be computed using the structure
relations \eqref{pill} and
 formulas \eqref{polli}.

In order to compute the four polynomials $P^v_{-3}$, $P^v_{-2}$, $P^v_{-1}$, and $P^v_{0}$
associated with the stratum $\mathfrak g_0$ of
$\mathrm{Prol}(\mathfrak g)$, 
we have to choose a basis $X_{-3}, X_{-2}, X_{-1}, X_0$
of $\mathfrak g_0$.
We can identify $\mathfrak g_{0}$ with $\Endo(\mathfrak g_{1})$, and
$\mathfrak g_1$ with $\R^2$ via the basis $X_1,X_2$.
Hence, we can make the
following choice:
\[
   X_{-3} = \left(
\begin{array}{cc}
 0 &  1
\\
  0 &  0 
\end{array}
\right),\quad
   X_{-2} = \left(
\begin{array}{cc}
 0 &  0
\\
  0 &  1
\end{array}
\right),\quad
   X_{-1} = \left(
\begin{array}{cc}
 1&  0
\\
  0 &  0
\end{array}
\right),\quad
   X_{0} = \left(
\begin{array}{cc}
 0 &  0
\\
  1 &  0
\end{array}
\right).
\]
The commutator $[X_j,X_i]$, with $j=-3,\ldots,0$ and $i=1,2$,
is identified with $X_j(X_i)$ where $X_j$ is a linear operator
on $\mathfrak g_1$.  
The polynomials associated with this basis are the following
\begin{equation}\label{537}
 \begin{split}
  P_{-3}^v (x)  & = v_4 (x_3 x_2 +x_5)
- v_5 \frac{x_2^{3}}{6} 
+v_6(x_4 x_2 +x_7)
+v_7(x_5 x_2 +2x_8)
+ v_8\frac{x_2^{4}}{24} ,
\\
 P_{-2}^v(x) &  
= v_4 x_4 + v_5(x_2 x_3 + 2 x_5) + v_6 x_6+v_7( x_2 x_4 + 2 x_7)
 +v_8( x_2
x_5 + 3 x_8),
\\
P_{-1}^v(x)  &= v_4(x_1 x_3 + 2 x_4) +v_5\big(x_5 -\frac{x_1 x_2^2}
{2}  \big) +v_6( x_1 x_4 + 3 x_6) 
\\
&\qquad +v_7( x_1 x_5 + 2 x_7) 
+v_8\big(\frac{x_1 x_2^3}{6}
 + x_8\big),
\\
 P_{0}^v (x)  &= 
v_4\frac{x_1^3}{6}
+v_5 \big(\frac{x_1^2 x_2}{2} + x_1 x_3 + x_4\big)- v_6
\frac{x_1^4}{24}
\\
&\qquad
 +v_7\big(-\frac{x_1^3 x_2}{6} + x_1 x_4 + 2 x_6\big)
 +v_8\big( -\frac{x_1^2 x_2^2}{4} + x_1 x_5 + x_7\big).
 \end{split}
\end{equation}
These formulas can be checked in the following way.
According to \eqref{derivatepolinomi}, the polynomials \eqref{537}
have to satisfy
the   structure identities
\begin{equation}\label{ID}
\begin{split}
  X_1 P_{-3}^v &= 0,\quad
  X_1 P_{-2}^v = 0,\quad
  X_1 P_{-1}^v = P_1^v,\quad
  X_1 P_{0}^v = P_2^v,
\\
  X_2 P_{-3}^v &= P_1^v,\quad
  X_2 P_{-2}^v = P_2^v,\quad
  X_2 P_{-1}^v = 0,\quad
  X_2 P_{0}^v = 0.
\end{split}
\end{equation}
Notice that we have the structure constants
$c_{-3,1}^k = c_{-2,1}^k = c_{-1,2}^k = c_{0,2}^k=0$, 
$c_{-3,2}^k= c_{-1,1}^k =\delta_{1k}$, 
$c_{-2,2}^k= c_{0,1}^k = \delta _{2k}$, $k=1,2$.
Along with the condition $P_j^v(0)=0$, $j=-3,\ldots,0$, the  identities
\eqref{ID} uniquely determine the polynomials. 
Now, by a direct computation based on \eqref{CAMP}, it can be checked that the
polynomials \eqref{537} do satisfy \eqref{ID}.

The polynomials \eqref{537} can also be computed  
using the following algorithm. Let $\gamma:[0,1]\to\R^8$
be a horizontal curve such that $\gamma(0)=0$ and, for $i,j=1,2$, let
\[
 B_{ij}^v (t) = \int_0^t P_i^v(\gamma)\dot\gamma_j\, ds,\quad t\in[0,1].
\]
Using the formulas \eqref{51} for $P_1^v$ and $P_2^v$,
the integrals in the right hand side can be computed.
Using the explicit formulas for the coefficients of the vector field $X_2$ in
\eqref{CAMP}, the resulting
functions can be shown to be  polynomials of the coordinates of $\gamma(t)$. Namely, along
$\gamma$ we have
$B_{12}^v =P_{-3}^v$, $B_{22}^v=P_{-2}^v$, $B_{11}^v = P^v_{-1}$, and $
B_{21}^v = P_0^v$.

\medskip 
 
Let $\gamma:[0,1]\to G=\R^8$ be an abnormal extremal such that $\gamma(0)=0$.
By Theorem \ref{cor:polin} (see also Remark \ref{R2}) there
exists $v\in\R^8$, with $v\neq
0$, such that,  for all $j=-3,-2,\ldots,3$, we have
\begin{equation} \label{piox2}
    P_j^v(\gamma(t)) = 0,\quad t\in[0,1].
\end{equation}
In particular, $v=(v_1,\ldots,v_8)$ satisfies $v_1=v_2=v_3 =0$ (because
$\gamma(0)=0$).

For each $j=-3,\ldots, 3$, we define the following $5$-dimensional vector  of
polynomials
\[
 Q_j(x) = \big( Q_{j4}(x),\ldots , Q_{j8}(x)\big),
\]
where $Q_{jk}$ are defined in \eqref{pipox}. By \eqref{polli}, identity
\eqref{piox2}
reads
\[
  \sum_{k=4}^ 8 v_ k Q_{jk}(\gamma(t)) = 0,\quad t\in[0,1].
\]
In other words, along the curve $\gamma$ the seven $5$-dimensional vectors
$Q_{-3}(x),\ldots,Q_3(x)$ are orthogonal to the nonzero vector
$(v_4,\ldots,v_8)\in\R^5$. It follows that the $7\times5$ matrix
\[
 Q(x) = \left(
\begin{array}{ccc}
        Q_{-3,4}(x)&\ldots&  Q_{-3,8}(x)
\\
 \vdots &&\vdots
\\
        Q_{3,4}(x)&\ldots&  Q_{3,8}(x)
       \end{array}
 \right)
\]
 has rank at most $4$ along the curve $\gamma$.

Let $M_1(x),\ldots,M_{21}(x)$ be the $21$ different $5\times5$ minors of
$Q(x)$ and define the determinant functions 
\begin{equation}\label{centrox}
f_k(x) = \det (M_k(x)),\quad k=1,\ldots,21,\quad x\in \R^8. \end{equation}
It can be proved that the functions $f_1,\ldots, f_{21}$ are homogeneous
polynomials. By the discussion above, we
have
\[
    f_k(\gamma(t)) = 0 \quad \textrm{for all $t\in[0,1]$ and for all
$k=1,\ldots,21$}.
\]
The polynomials $f_1,\ldots, f_{21}$ do not depend on $v$.

\begin{teo} \label{5.2}
Let $G=\R^8$ be the free nilpotent Lie group of rank $2$ and step $4$. 
There exists a nonzero  homogeneous polynomial 
$Q:\R^8\to\R$ such that any abnormal curve $\gamma:[0,1]\to\R^8=G$ with
$\gamma(0)=0$ satisfies $Q(\gamma(t))=0$ for all $t\in[0,1]$. 
\end{teo}

\begin{proof}
It is enough to show that at least one of the polynomials 
$f_1,\ldots,f_{21}$ in \eqref{centrox} is nonzero.  Let us consider the
polynomial with least degree 
\[
   Q(x) = \det M(x) = \det \left(
\begin{array}{ccc}
        Q_{-1,4}(x)&\ldots&  Q_{-1,8}(x)
\\        Q_{0,4}(x)&\ldots&  Q_{0,8}(x)
\\

 \vdots &&\vdots
\\
        Q_{3,4}(x)&\ldots&  Q_{3,8}(x)
       \end{array}
 \right).
\]
The homogeneous degree of $Q$ is $14$. 
We highlight the anti-diagonal of the matrix $M$:
\[
    M(x)  =\left(
\begin{array}{ccccc}
       *&*&*&*& x_1 x_2^3/6+x_8
\\ 
    *&*&*&-x_1^3 x_2/6 +x_1x_4 +2 x_6 &*
\\
*&*&x_4&*&*
\\
* & x_1 x_2+x_3 & *&*&*
\\
-x_1&*&*&*&*
\\
       \end{array}
 \right).
\]

Upon inspection of the polynomials $P_{-1}^v,\ldots,P_3^v$ in \eqref{51} and
\eqref{537}, we observe the following facts.
The variable $x_8$ appears in the  entry $M_{ij}$ 
only when $i=j=5$. When $i,j\leq 4$, the variable $x_6$
appears only in the entry $M_{ij}$    with  $i=j=4$.
When $i,j\leq 3$, the variable $x_4$
appears   only in the entry $M_{ij}$ with $i=j=3$.
When $i,j\leq 2$, the variable $x_3$
appears only in the entry $M_{ij}$ with $i,j=2$. Finally, we 
  have $M_{11}= -x_1$. It follows that
\[
  Q(x) = -2x_1x_3x_4 x_6x_8 +R(x),
\]
where $R(x)$ is a polynomial that does not contain the
monomial $x_1x_3x_4 x_6x_8 $. This proves that $Q\neq 0$.

\end{proof}

\begin{remark}
Any abnormal curve passing through $0$ is in the intersection of the
zero sets of the $21$ polynomials \eqref{centrox} in $\R^8$. 
Even though all these polynomials are explicitly computable, the precise structure of this intersection is
not clear. 

For any $v\in\R^8$ with $v\neq 0$ and $v_1=v_2=v_3$,
the equation $P_3^v(x)=0$ determines two (in some exceptional cases four)
different abnormal extremals $\gamma^v$   parameterized by arc-length and such
that $\gamma^v(0)=0$. The mapping $v\mapsto
\gamma^v(1)$ seems   to parametrize a
subset of $\R^8$ that is  at most $5$-dimensional. The abnormal set of $G=\R^8$
is then presumably $6$-dimensional, while in Theorem \ref{5.2} abnormal
extremals are shown to be in a set with dimension less than or equal to $7$. 
\end{remark}

\section{Spiral-like Goh extremals} 
\label{Goh}

The shortening technique introduced in \cite{LM} has
two steps:  a curve with a corner at a given point is blown up; the
limit curve obtained in this way
is shown not to be length minimizing. This provides an
``almost-$C^1$''
regularity result for sub-Riemannian length minimizing curves.
The technique, however, fails when, in the
blow up at the singular point,
the   curve has no proper angle.
Here,  we show that there do exist Goh extremals of this kind.
The example is a generalization of \cite[Example 6.4]{LLMV}.

Let $G$ be an $n$-dimensional stratified nilpotent Lie group with Lie algebra
$\mathfrak g = \mathfrak g_1\oplus\cdots\oplus \mathfrak g_s$, 
where $s\geq 3$ is step of
the group. 
An abnormal
extremal $\gamma:[0,1]\to G$ with dual curve $\lambda = \lambda_1\vartheta_1
+\cdots+\lambda_n\vartheta _n$ is said to be a Goh extremal if $\la_i=0$ for all
$i=1,\ldots,n$ such that 
$d(i)=1$ or $d(i)=2$.
By Theorem \ref{teointduale},  
a horizontal curve $\gamma:[0,1]\to G$ with $\gamma(0)=0$ is a  Goh extremal
precisely when there exists $v\in\Rn$, $v\neq0$,
such that
\begin{equation} \label{GOE}
 P_i^v(\gamma(t)) = 0 ,\quad \textrm{for all $t\in[0,1]$ and
$d(i)\in\{1,2\}$.}
\end{equation}
By Theorem \ref{propderivpolinomi}, the condition $P_i^v(\gamma) = 0 $ for  
$d(i)=2$ is equivalent to \eqref{GOE}.

Let $F$ be the  free nilpotent Lie group of rank 3 and step 4 and consider
the direct product $G = F\times F$. 
As $F$ is   diffeomorphic to $\R^{32}$, then  $G$ is a
stratified Lie group of rank 6 and step 4 diffeomorphic to $\R^{64}$.
We fix a basis $X_1,\ldots,X_{64}$  of $\mathrm{Lie}(G)$ adapted
to the stratification and we identify $G$ with $\R^{64}$ via exponential
coordinates of the second type.
We reorder and relabel the basis as $Y_{1},\dots,Y_{32},Z_{1},\dots,Z_{32}$,
where $Y_1,\ldots, Y_{32}$ is an adapted basis of a first copy of $\mathrm{Lie}
(F)$ and
$Z_1,\ldots, Z_{32}$ is an adapted basis of a second copy of $\mathrm{Lie} (F)$.
We denote the  corresponding coordinates on $G$ by $(y,z)\in\R^{64}$
with $y,z\in \R^{32}$.

Notice that  we have $[Y_{i},Z_{j}]=0$ 
for all $i,j=1,\dots,32$.
Then any abnormal polynomial $P_i^v(x)$ 
of $G$, $i=1,\ldots,64$
and $v\in\R^{64}$,  splits as
\[
 P_i^v(x) = Q_j^{v_y}(y)+Q_k^{v_z}(z), \quad x = (y,z),
\]
for some $j,k=1,\ldots,32$ and $v_y,v_z\in\R^{32}$, where
$Q_1^v ,\ldots,Q_{32}^v $, $v\in\R^{32}$, are the extremal polynomials of $F$.
In particular, for $i=4,5,6$ we have the polynomials: 
\[
\begin{split}
 Q_{4}^v(y)  & =  v_4  -y_1 v_7-y_2 v_8 
-y_3 v_9 + \frac{y_1^{2}}{2} v_{15}
+y_1 y_2 v_{16} 
+y_1 y_3 v_{17} 
+ \frac{y_2^{2}}{2} v_{18}
\\
&\qquad +y_2 y_3 v_{19}
+ \frac{y_3^{2}}{2} v_{20}+y_5v_{30}
+y_6v_{31}
\\
Q_{5}^v(y)  & = v_5 -y_1 v_{10}-y_2 v_{11} 
-y_3 v_{12} 
+ \frac{y_1^{2}}{2} v_{21}
+y_1 y_2 v_{22} 
+y_1 y_3 v_{23} 
+ \frac{y_2^{2}}{2} v_{24}
\\
&\qquad +y_2 y_3 v_{25}
+ \frac{y_3^{2}}{2} v_{26}
-y_4v_{30}
+y_6v_{32}
\\
Q_{6}^v(y)  & = v_6+ y_1 (v_9
- v_{11})
-y_2 v_{13} 
-y_3 v_{14} 
 - \frac{y_1^{2}}{2} v_{17}-y_1 y_2 v_{19} 
-y_1 y_3 v_{20} 
+\frac{y_1^{2}}{2} v_{22}  
\\
&\qquad
+  y_1 y_2 v_{24} 
+ y_1 y_3 v_{25}
+ \frac{y_3^{2}}{2} v_{29}
+y_1^2 v_{30}
+ (y_1 y_2-y_4)v_{31}  -y_5v_{32}
.
\end{split}
\] 
We omit the  deduction of these formulas.

Let $v,w\in\R^{32}$ be such that $v_i=w_i=0$ for all $i=1,\ldots,6$ and 
$|v|+|w|\neq0$. 
By Theorem \ref{teointduale}, any horizontal curve $\gamma =
(\gamma_y,\gamma_z):[0,1]\to G$ such that $\gamma(0)=0$ and  with support contained in the algebraic set  
\[
 \Sigma_{v,w} = \big  \{ (y,z)\in\R^{64}: \ Q_i^v(y) =Q_{i}^{w}(z)=0 \text{
for }i=4,5,6 \big\}
\]
is a Goh extremal. Notice that, by \eqref{derivatepolinomi}, $Q_i^v(\gamma_y)
=Q_{i}^{w}(\gamma_z)=0$ for $i=4,5,6$ implies
 $Q_i^v(\gamma_y)
=Q_{i}^{w}(\gamma_z)=0$ also for $i=1,2,3$. 
We choose $v\in\R^{32} $ such that  $v_7 = 1$,
$v_{18}= 2$, and $v_j =0$ otherwise. Then we have 
\[
Q^v_4(y) =  y_2^2-y_1,\quad Q_5^v(y)=Q_6^v(y)=0,\quad y=(y_1,\ldots,y_{32}).
\]
For any pair of Lipschitz functions
$\phi,\psi:[-1,1]\to\R$ with $\phi(0)=\psi(0)=0$, the
horizontal curve
$\gamma=(\gamma_y,\gamma_z):[-1,1]\to\R^{64}$ defined by
$\gamma_y(t) =(t^2,t,\phi(t),*,\ldots,*)$ and
$\gamma_z(t) = (t^2,t,\psi(t),*,\ldots,*)$
has support contained in 
$\Sigma_{v,v}$
and therefore it is a Goh extremal.

In particular, as in \cite{GMT} we can choose
\[
\phi(t) = t\cos(\log(1-\log |t|)),\quad \psi(t) = t\sin(\log(1-\log
|t|)), \quad t\in[-1,1],
\]
where  $\phi(0) = \psi(0) = 0$ in the limit sense. The curve $\gamma$ is  a kind
of  rectifiable spiral. Notice that 
$\max(|\phi'(t)|,|\psi'(t)|) \leq 2$ for all $t\in (0,1)$. Moreover, 
for any $\alpha\in[0,2\pi)$ there exists a positive and infinitesimal sequence
$(\lambda_k)_{k\in \N}$ such that 
\[
  \lim_{k\to\infty} \frac{1}{\lambda _k}  \big( \phi(\lambda_k t),
\psi(\lambda_k t)\big) = t(\cos\alpha, \sin \alpha),\quad t\in\R.
\]
Viceversa, if the limit in the left hand side does exist for some
infinitesimal sequence $(\lambda_k)_{k\in \N}$, then it is of the given form for
some $\alpha\in[0,2\pi)$.

The curve $\gamma$ is not   $C^1$  at $t=0$ and all its  tangents at $t=0$ are lines. 
Even though 
the spiral $\gamma$ is not likely to be length
minimizing near $t=0$, the shortening technique of \cite{LM} cannot be
applied  to prove this.


\end{document}